\documentclass[conference, 10 pt, twocolumn]{ieeeconf}  

\IEEEoverridecommandlockouts                              
\usepackage{bbm}
\usepackage{enumerate}
\usepackage{times}
\usepackage{textcomp}
\usepackage{cite}
\usepackage{url}
\usepackage{subfigure}
\usepackage{amsfonts,mathrsfs}
\usepackage{amssymb,amsmath}
\usepackage{verbatim}
\usepackage{acronym}
\usepackage{mathtools}

\usepackage{graphicx}
\usepackage{hyperref}

\usepackage{amsmath,amssymb}
\usepackage{algorithm}
\usepackage{algorithmic}
\usepackage{romannum}

\DeclareMathOperator{\E}{\mathbb{E}}

\newcommand{\Nc}{\mathcal{N}}
\newcommand{\Oc}{\mathcal{O}}
\newcommand{\R}{\mathbb{R}}
\newcommand{\eps}{\varepsilon}
\newcommand{\F}{\mathcal{F}}

\newcommand{\x}{\boldsymbol{x}}
\newcommand{\y}{\boldsymbol{y}}
\newcommand{\z}{\boldsymbol{z}}
\newcommand{\step}{\eta}
\newcommand{\fuj}{f_j^u}
\newcommand{\nbli}{\nabla^i}

\newcommand{\Dj}{D_j}
\newcommand{\zi}{\boldsymbol{z}^i}
\newcommand{\tauijt}{\tau_j^i(t)}
\newcommand{\grad}{\boldsymbol{g}}
\newcommand{\Ft}{\F_\tau}
\newcommand{\Rdi}{\R^{d_i}}

\newtheorem{theorem}{Theorem}

\newtheorem{assumption}{Assumption}

\newtheorem{lemma}{Lemma}

\begin{document}
\title{Zeroth-Order {Non-Convex} Optimization for Cooperative Multi-Agent Systems with Diminishing Step Size and Smoothing Radius} 
\author{Xinran Zheng, Tara Javidi, and Behrouz Touri\thanks{This research is supported by NSF grant CNS-2148313, AutoCOMBOT MURI Grant, and AFOSR FA9550-23-1-0057, emails: \{x7zheng, tjavidi, btouri\}@ucsd.edu}\\
{\small {Electrical and Computer Engineering Department, University of California San Diego }}}
\maketitle
\begin{abstract}
We study a class of zeroth-order distributed optimization problems, where each agent can control a partial vector and observe a local cost that depends on the joint vector of all agents, and the agents can communicate with each other with time delay. We propose and study a gradient descent-based algorithm using two-point gradient estimators with diminishing smoothing parameters and diminishing step-size and we establish the convergence rate to a first-order stationary point for general nonconvex problems. A byproduct of our proposed method with diminishing step size and smoothing parameters, as opposed to the fixed-parameter scheme, is that our proposed algorithm does not require any information regarding the local cost functions. This makes the solution appealing in practice as it allows for optimizing an unknown (black-box) global function without prior knowledge of its smoothness parameters. At the same time, the performance will adaptively match the problem instance parameters.
\end{abstract}

\thispagestyle{empty}
\pagestyle{empty}
\section{Introduction}
Distributed optimization has found broad applications, such as distributed learning \cite{omidshafiei2017deep}, network source allocation \cite{tychogiorgos2013non}, and wind farm control \cite{marden2013model}. Generally, such problems are related to minimizing the sum $\sum_{i=1}^nf_i(\x)$ of local functions $f_i(\x)$s where $f_i(\x)$ is only known to agent $i$. There are two classes of distributed optimization problems. Extensive research has been done on consensus-based distributed optimization, where each agent controls the full vector $\x$ and needs to achieve consensus with other agents on an optimal or stationary decision vector~\cite{nedic2009distributed,shi2015extra,tatarenko2017non,zeng2018nonconvex,Reisizadeh2022}. This paper investigates the other class of problems, which we call cooperative multi-agent systems. In these problems, each agent can only partially control the decision vector, but the local cost $f_i$ observed by agent $i$ reflects the impact of all agents' decisions.

In many real-world situations, objective functions are available only as the output of a black-box, or the relationship between the objective function and the variables is so complicated that the direct calculation of the derivatives could be expensive or infeasible. This has led to a research interest in black-box optimization. The well-known Kiefer-Wolfowitz scheme~\cite{kiefer1952stochastic} uses $2d$ ($d$ is the dimension of the variable) function evaluations to construct a gradient estimator, but this does not scale up with high-dimensional problems. Therefore, many existing works propose and analyze two-point~\cite{nesterov2017random} and single-point~\cite{flaxman2004online} gradient estimators and show that the convergence rates of two-point estimator algorithms are comparable to their first-order counterparts~\cite{nesterov2017random}.

Recently zeroth-order methods have been investigated to solve cooperative multi-agent problems. For example, \cite{tang2023zeroth} and~\cite{tang2023zerothjournal} proposed algorithms based on two-point gradient estimator and analyzed the convergence rates in both convex and non-convex, and both noiseless and noisy settings. In~\cite{shen2021asynchronous} an asynchronous method is proposed  based on single-point gradient estimator and showed that their method outperforms two-point methods under asynchronous updating. However, both papers use constant step-size and smoothing radius that depend on the total number of iterations and the smoothness and Lipschitz parameters of objective functions, which might not be known in practical applications.

Inspired by \cite{tang2023zeroth}, we analyze a cooperative multi-agent optimization problem with a (slightly) more generalized communication scheme between agents. As opposed to~\cite{tang2023zeroth}, we propose to use diminishing step-size and smoothing radius schedules{, which have been used in a similar manner in \cite{drusvyatskiy2022improved} and \cite{bravo2018bandit}}. This leads to convergence guarantees without any prior knowledge of the parameters of objective functions (which are required in~\cite{tang2023zeroth}). This also establishes a rigorous asymptotic bound that was not established previously. We demonstrate that for a class of diminishing step-sizes and smoothing radius, our algorithm can obtain $\Oc(T^{-(\frac{1}{2}-\eps)})$ convergence rate, for  sufficiently small $\eps>0$.\\
\textbf{Notations}. {We denote the set of real numbers by $\R$ and the vector space of $d$-dimensional real-valued row vectors by $\R^d$. We use bold letters to denote row vectors. We use $I_d$ to denote the identity matrix of dimension $d$. We use $\Nc(0, \boldsymbol{\Sigma})$ to denote the multi-dimensional Gaussian distribution with zero mean and covariance $\boldsymbol{\Sigma}$.} We use $[n]$ to denote $\{1, \ldots, n\}$. We use $\| \cdot \|$ to denote the standard Euclidean norm, and $\langle\cdot,\cdot\rangle$ to denote the standard Euclidean inner product. 

\section{Problem Formulation}
We consider a cooperative multi-agent optimization problem among $n$ agents, where the agents seek to find
\begin{align}\label{eqn:cooperative}
    \min_{\x^i \in \Rdi} f(\x^1,\ldots,\x^n):=\frac{1}{n}\sum_{i=1}^n f_i(\x^1,\ldots,\x^n).
\end{align}
We assume that $\x^i \in \R^{d_i}$ is the decision vector controlled by agent $i\in[n]$, and $f_i: \R^{d_1} \times \cdots \times \R^{d_n} \rightarrow \R$ is the local cost function observed by agent $i\in[n]$.


In this paper, we consider the setting that each agent $i\in[n]$ can only access the (noiseless) value of its local cost function $f_i$, and the information regarding higher order derivatives of $f_i$ is unavailable to the agent. In addition, agents are able to exchange information over a possibly time-varying network. Therefore, we assume that the exchange of information incurs a delay. {More precisely, for any two agents $i,j \in [n]$, when agent $i$ receives some information from agent $j$ (with some delay), it also receives a time stamp that denotes the time when agent $j$ sends this information.} 

\section{The Main Results}
In this section, we introduce our algorithm to solve problem (\ref{eqn:cooperative}). Our algorithm operates on the discrete-time instants $t=0,1,2,\ldots$. At $t=0$, each agent initializes a guess at an arbitrary point $\x^i(0)=\x^i_0\in \Rdi$. At time $t=0,1,2,\ldots$, each agent $i\in[n]$ generates a random vector $\zi(t) \sim \Nc (0,I_{d_i})$, adjusts its decision vector to be ${\x^i(t)+u(t)\zi(t)}${, where $u(t)$ is a positive scalar we refer to as \textit{smoothing radius}}. Once, every agent does that, agent $i\in[n]$ observes the corresponding local cost $f_i\left(\x(t)+u(t)\z(t)\right)$, where $\x(t)$, $\z(t)$ are the concatenated vectors $(\x^1(t),\ldots,\x^n(t))$, $(\z^1(t),\ldots,\z^n(t))$, respectively. Then, in the same way, each agent $i$ adjusts its vector to be $\x^i(t)-u(t)\zi(t)$, and observes the corresponding local cost $f_i(\x(t)-u(t)\z(t))$. Then, using these two observations,  the agent $i$ computes the approximate derivative of $f_i(\x(t))$ along the vector $\z(t)$
\begin{align*}
    D_i(t)=\frac{f_i(\x(t)+u(t)\z(t))-f_i(\x(t)-u(t)\z(t))}{2u(t)}
\end{align*}
and sends it {out}. At the same time, each agent $i$ may receive the derivative information of agent $j \in [n]$, and only stores the latest  $\Dj(t)$ and the latest time stamp from each agent $j$. To be specific, at each time $t$, we denote the latest time $s$ that $\Dj(s)$ (calculated by agent $j$) is received by agent $i$ by time $t$ to be $\tauijt$. Thus, $t-\tauijt \geq 0$ represents  the communication time delay between agent $j$ and $i$. If agent $i$ has not received any $\Dj$ from agent $j$, we denote $\tauijt=-1$. For $i=j$, we have $\tauijt=t$. Then, each agent $i$ computes a partial gradient estimator as
\begin{align*}
    \grad^i(t)=\frac{1}{n}\sum_{j:\tauijt \geq 0}{\Dj(\tauijt)\zi(\tauijt)}.
\end{align*}
Finally, each agent $i$ updates its decision vector as
\begin{align}\label{eqn:maindyn}
    \x^i(t+1)=\x^i(t)-\step(t)\grad^i(t).
\end{align}


\subsection{Analysis}
Here, we present our main result concerning  the convergence guarantee for our algorithm.
In order to do so, we need to assume some regularity conditions and assumptions on various objects in our work. First, we discuss the assumptions on local cost functions $f_i(\cdot)$.
\begin{assumption}[Assumption on Objective Functions]
\label{asm:func}
    We assume ${f^*:=\inf_{\x \in \R^d} f(\x)>-\infty}$. Further, for $i\in [n]$, $f_i: \R^d \rightarrow \R$ is $G$-Lipschitz and $L$-smooth, i.e.,
    \begin{align}\label{eqn:lipschitz}
        |f_i(\x)-f_i(\y)| &\leq G\| \x-\y \|,\\\label{eqn:smooth}
        \|\nabla f_i(\x)-\nabla f_i(\y)\| &\leq L\| \x-\y \|,
    \end{align}
    for all $\x,\y \in \R^d$. 
\end{assumption}
{We may comment that the above assumptions can be replaced by requiring that the trajectories of the dynamics are bounded in a set $E$, and the $f_i$s satisfy the above conditions only on the set $E$. }

The following is the assumption on the communication delay between the agents. 
\begin{assumption}[Assumption on Communication Delay]
\label{asm:delay}
    {The communication delay between agents $i,j\in[n]$ is a uniformly bounded non-negative integer, i.e., $0 \leq t-\tauijt \leq B$ for some constant $B$.}
\end{assumption}

Note that Assumption \ref{asm:delay} subsumes the assumption on time delay in \cite{tang2023zeroth}, where information is flooding over a (time-invarying) network connecting all agents and there, the time delay between each pair of agents relates to their distance over the network. {Here, the network can be time-varying and we only need bounded delay $B$ between any two nodes when the information is flooded over the network. This is  equivalent to the strongly $B$-connected assumption that is commonly assumed in the distributed optimization literature (see e.g. \cite{nedic2009distributed,tatarenko2017non}).}

Finally, we discuss the assumption regarding the step-size sequences and the smoothing radius sequences. 
\begin{assumption}[Assumption on Step-size Sequences]
\label{asm:step}
    The step-size $\step(t)$ and the smoothing radius $u(t)$ take the form $\step(t)=\frac{\step_0}{(t+1)^\alpha}$, and $u(t)=\frac{u_0}{(t+1)^\beta}$, respectively, where $\step_0 > 0$, $u_0 > 0$, $\alpha$ and $\beta$ are positive exponents satisfying $0<\alpha<1$ and $\beta>0$.
\end{assumption}

{A practical and important caveat here is that in practice, if $u(t)$ is too small, noise could dominate the function difference, which leads to a bad approximation of the gradient \cite{liu2020primer}.}

We are now ready to present our main result that characterizes the convergence rate of our algorithm.

\begin{theorem}\label{Theorem1}
    Under assumptions \ref{asm:func}-\ref{asm:step},
    \begin{align*}
        \lim_{T\to\infty}\frac{1}{T+1}\sum_{t=0}^{T}{\E[\|\nabla f(\x(t))\|^2]}=0.
    \end{align*}
    In particular, if $\alpha=\frac{1}{2}$, $\beta=\frac{1}{4}$, for any $0\!<\!\eps\!<\frac{1}{2}$, we have
    \begin{align*}
        \frac{1}{T+1}\sum_{t=0}^{T}{\E[\|\nabla f(\x(t))\|^2]}=\Oc (T^{-(\frac{1}{2}-\eps)}).
    \end{align*}
    
\end{theorem}

In a closely related work~\cite{tang2023zeroth},  the step-size and smoothing radius were assumed to be constants given by $\step=\frac{\alpha_\step}{L\sqrt{\Bar{b}\sqrt{n}d}\sqrt{T-B+1}}$, where $G,L$ are the Lipschitz and smoothness parameters of $f_i$s, and the number of iterations $T$ is predetermined.  {They show that for these fixed parameters, given the number of iterations $T>B$, they have $\frac{1}{T-B+1}\sum_{t=B}^{T}{\E[\|\nabla f(\x(t))\|^2]}\leq K \sqrt{\frac{\Bar{b}\sqrt{n}d}{T-B+1}}$ for some constant $K>0$}. Our diminishing step-size approach achieves $\Oc (T^{-(\frac{1}{2}-\eps)})$ where $\eps>0$ can be arbitrarily close to 0. Therefore, we show that using diminishing step-size does not slow down the convergence in a significant manner. However, in our result, we do not need to know any of these parameters to design our step-size and smoothing radius.

{
\section{Numerical Experiment}
Here, we demonstrate the performance of our algorithm on a wind farm power maximization problem studied in \cite{tang2023zeroth,marden2013model}. This example consists of $n=80$ wind turbines, each of which can adjust its own axial induction factor denoted by decision variable $x^i\in \R$. The power generated by turbine $i$, denoted by $f_i$, depends on its own axial induction as well as those of the wind turbines upstream. We denote $\x=(x^1,\ldots,x^n)$, so the wind farm power maximization problem can be written as ${\max_{\x}f(\x)=\frac{1}{n}\sum_{i=1}^n f_i(\x)}$.

In our experiment, we start from an initial point ${\x(0)=(\frac{1}{3},\ldots,\frac{1}{3})}$, and normalize the power $f(\x)$ by the optimal power $f^*=f(\x^*)$, where $\x^*$ is the optimal action profile. We compare the performance of our algorithm with the algorithm using constant $\step$ and $u$ proposed in \cite{tang2023zeroth}. For our algorithm, we choose $\step_0=0.1$, $u_0=0.01$, and two different $(\alpha,\beta)$ pairs $(\alpha,\beta)=(0.4,0.5)$ or $(\alpha,\beta)=(0.51,0.25)$. For constant-step-size algorithm in \cite{tang2023zeroth}, we choose three different values for step size, $\step_1=0.05$, $\step_2=0.01$, $\step_3=0.005$, both with $u=0.001$. We run the algorithm for $T=8000$ iterations and repeat this for 10 trials.

\begin{figure}[h]
    \centering
    \includegraphics[width=0.8\linewidth]{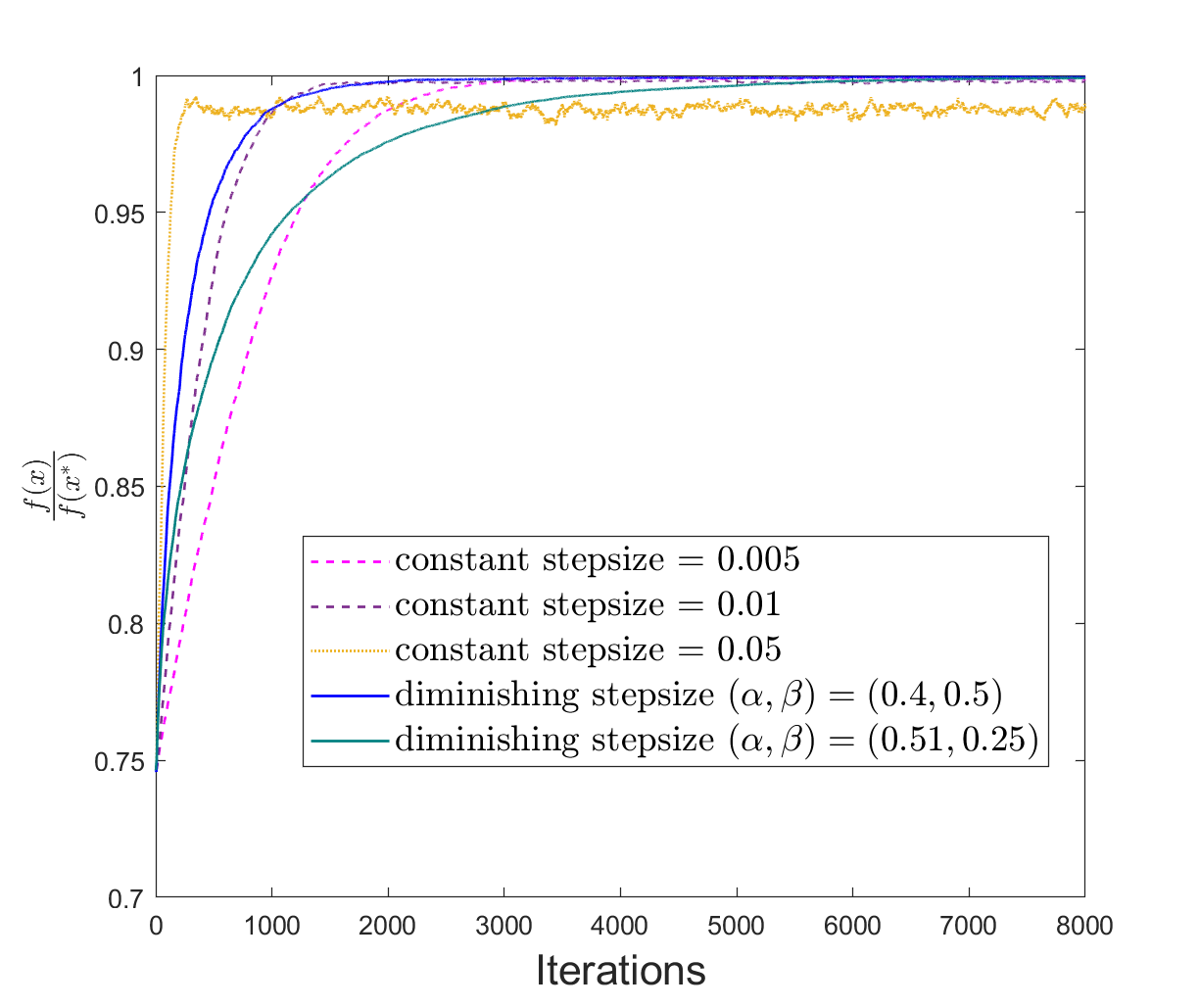}
    \caption{Optimal wind farm power extraction using diminishing step size vs constant step size algorithm.}
    \label{fig:simulation}
\end{figure}

In Figure \ref{fig:simulation}, we note that for a constant step-size algorithm, once we choose a step size $\step$, there exists a $T'$ such that when $t>T'$, $\x(t)$ will stay in a neighborhood of $\x^*$. When the constant step size $\step$ is large, this neighborhood is large. When $\step$ is small, this neighborhood is small, but the convergence to the neighborhood is slow. By contrast, for our diminishing step size algorithm, we observe convergence to the optimal point and choose proper parameters to achieve a balance between accuracy and convergence rate.

}

\section{Proof of Main Result}\label{sec:proof}
To prove the main result, first we establish a recursive inequality for $\E[f(\x(t))]$. Using that, we can bound $\sum_{t=B}^T \step(t)\E[\| \nabla f(\x(t))\|^2]$. Finally, we translate such a bound to a bound for the (empirical) average expected squared norm of the gradients, i.e., $\frac{1}{T+1}\sum_{t=0}^{T}{\E[\|\nabla f(\x(t))\|^2]}$. We establish this using intermediate results as discussed below.

First, we show that our gradient estimator is nearly unbiased and its proof follows from a similar argument as the derivation of (21) in \cite{nesterov2017random} and the proof of Lemma 5(b)-\cite{malik2019derivative}.

\begin{lemma}\label{lemma:gradient_estimator}
    {Consider the \textit{smoothed version} of $f$ and $f_j$ defined by $f^u(\x):=\E_{\z}[f(\x+u\z)]$ and ${\fuj(\x):=\E_{\z}[f_j(\x+u\z)]}$, respectively, where $\z\sim \Nc(0, I_d)$. 
    Then, if $f_j$ is $G$-Lipschitz and $L$-smooth (see Assumption~\ref{asm:func}) for $j\in[n]$, so will be the smoothed function $\fuj:\R^d \rightarrow \R$. Moreover, for each $\x \in \R^d$,
    \begin{align}
    \label{eqn:eqn1_of_lemma:gradient_estimator}
        \E_{\z}\left[\frac{f_j(\x+u\z)-f_j(\x-u\z)}{2u}\z\right]&=\nabla \fuj(\x),\text{ and }\\
    \label{eqn:eqn2_of_lemma:gradient_estimator}
        \| \nabla f(\x)-\nabla f^u(\x)\| &\leq uL\sqrt{d}.
    \end{align}}
\end{lemma}

Our next intermediate result provides a bound on the second moment of our gradient estimator{, whose proof follows the proof of Lemma 3 in \cite{tang2023zeroth} and Lemma 3 in \cite{tang2023zerothjournal}}.
\begin{lemma}
\label{lemma:2norm_grad_bound}
    Let $\z \sim \Nc(0, I_d)$, and assume $f:\R^d \rightarrow \R$ is $G$-Lipschitz. Then
    \begin{align*}
        \E_{\z} \left[ \left|\frac{f(\x+u\z)-f(\x-u\z)}{2u}z_i\right|^2 \right] \leq 2\sqrt{6}G^2,
    \end{align*}
    for any $i=1,\ldots,d$, where $z_i$ denotes the $i$'th entry of $\z$. {In particular, if for all $j\in[n]$, $f_j$ is $G$-Lipschitz, then for any $t \geq 0$ in the evolution of our algorithm, we have
    \begin{align*}
        \E[\| \Dj(\tauijt)\zi(\tauijt)\|^2] &\leq 2\sqrt{6}G^2d_i,\\
        \E[\| \grad (t) \|^2] &\leq 2\sqrt{6}G^2d.
    \end{align*}
    
    }
\end{lemma}

Our next result quantifies the effect of the delay on our algorithm.
\begin{lemma}
\label{lemma:effect_of_delay}
    {Suppose that the assumptions of Theorem~\ref{Theorem1} holds. } Then, for any $i,j \in [n]$ and any $t\geq B$, we have 
    \begin{align*}
        &\E[\|\nbli\fuj(\x(t))-\nbli\fuj(\x(\tauijt))\|^2]\\
        &\qquad\qquad \leq2\sqrt{6}G^2L^2Bd\sum_{\tau=-B}^{-1}{\step^2(t+\tau)},
    \end{align*}
    and similarly, $\E[\|\nbli f(\x(t))-\nbli f(\x(\tauijt))\|^2]\leq 2\sqrt{6}G^2L^2Bd\sum_{\tau=-B}^{-1}{\step^2(t+\tau)}$.
\end{lemma}

\begin{proof} 
Since $\fuj$ is $L$-smooth, we have
\begin{align*}
    &\E[\|\nbli\fuj(\x(t))-\nbli\fuj(\x(\tauijt))\|^2]\\
    \leq& L^2\E[\|\x(t)-\x(\tauijt)\|^2]\\
    \leq& L^2\E[(\sum_{\tau=-B}^{-1}{\|\step(t+\tau)\grad(t+\tau)\|})^2]\\
    \leq& L^2\sum_{\tau=-B}^{-1}{\step^2(t+\tau)}\sum_{\tau=-B}^{-1}{\E[\|\grad(t+\tau)\|^2]}\\
    \leq& 2\sqrt{6}G^2L^2Bd\sum_{\tau=-B}^{-1}{\step^2(t+\tau)}.
\end{align*}
Here, the penultimate inequality is due to the Cauchy-Schwartz inequality (3I in \cite{strang2006linear}), and the last inequality uses Lemma \ref{lemma:2norm_grad_bound}.
The second inequality follows similarly.
\end{proof}

Next, we introduce Lemma \ref{lemma:innerproduct_bound1}, \ref{lemma:innerproduct_bound2}, and \ref{lemma:innerproduct_bound3} to bound $\E [-\sum_{i=1}^n\langle \nbli f(\x(t)), \grad^i(t) \rangle]$.

\begin{lemma}
\label{lemma:innerproduct_bound1}
    Let $p(t)=\sqrt{\sum_{\tau=-B}^{-1}{\step^2(t+\tau)}}$. Then for any $t\geq B$, we have
    \begin{align*}
        &\E[-\frac{1}{n}\sum_{i,j}{\langle \nbli f(\x(t))-\nbli f(\x(\tauijt)),\Dj(\tauijt)\zi(\tauijt)\rangle}]\\
        &\leq 2\sqrt{6}G^2Ld\sqrt{nB}p(t).
    \end{align*}
\end{lemma}

\begin{proof} 
{Using $ab\leq \frac{1}{2}(a^2+b^2)$,} we have
\begin{align*}
    &\E[-\frac{1}{n}\sum_{i,j}{\langle \nbli f(\x(t))-\nbli f(\x(\tauijt)),\Dj(\tauijt)\zi(\tauijt)\rangle}]\\
    \leq& \frac{1}{2n}\frac{1}{\sqrt{nB}Lp(t)}\sum_{i,j}{\E[\|\nbli f(\x(t))-\nbli f(\x(\tauijt))\|^2]}\\
        &+\frac{1}{2n}\sqrt{nB}Lp(t)\sum_{i,j}{\E[\|\Dj(\tauijt)\zi(\tauijt)\|^2]}\\
    \leq& \frac{1}{2n}\frac{1}{\sqrt{nB}Lp(t)}\sum_{i,j} \left(2\sqrt{6}G^2L^2Bd\sum_{\tau=-B}^{-1}{\step^2(t+\tau)}\right)\\
        &+\frac{1}{2n}\sqrt{nB}Lp(t)\sum_{i,j}{2\sqrt{6}G^2d_i}= 2\sqrt{6}G^2Ld\sqrt{nB}p(t),
\end{align*}
{where the last inequality follows from Lemma~\ref{lemma:effect_of_delay} and Lemma~\ref{lemma:2norm_grad_bound}.} 
\end{proof}

Similar to Lemma \ref{lemma:innerproduct_bound1}, the following result is used to prove Lemma \ref{lemma:innerproduct_bound3}. 
\begin{lemma}
\label{lemma:innerproduct_bound2}
    For any $t\geq B$, we have
    \begin{align*}
        &\E[-\frac{1}{n}\sum_{i,j}\langle \nbli f(\x(\tauijt))-\nbli f(\x(t)),\\
        &\qquad\qquad \nbli\fuj(\x(\tauijt))-\Dj(\tauijt)\zi(\tauijt)\rangle]\\
        &\leq 2\sqrt{6}G^2Ld\sqrt{nB}p(t).
    \end{align*}
\end{lemma}

\begin{proof}
Similar to the proof of Lemma \ref{lemma:innerproduct_bound1}, we have
\begin{align}
\label{eqn:proof_of_lemma:innerproduct_bound2}
    &\E[-\frac{1}{n}\sum_{i,j}\langle \nbli f(\x(\tauijt))-\nbli f(\x(t)),\\\nonumber
    &\qquad\qquad \nbli\fuj(\x(\tauijt))-\Dj(\tauijt)\zi(\tauijt)\rangle]\cr
    & \leq\frac{1}{2n}\frac{1}{\sqrt{nB}Lp(t)}\sum_{i,j}{\E[\|\nbli f(\x(t))-\nbli f(\x(\tauijt))\|^2]}\cr
        &+\frac{1}{2n}\sqrt{nB}Lp(t)\sum_{i,j}{\E[\|\Dj(\tauijt)\zi(\tauijt)-\nbli\fuj(\x(\tauijt))\|^2]}.\nonumber
\end{align}
Since $\E[\Dj(\tauijt)\zi(\tauijt)]=\nbli\fuj(\x(\tauijt))$, we have $\E[\|\Dj(\tauijt)\zi(\tauijt)-\nbli\fuj(\x(\tauijt)))\|^2]
\leq\E[\|\Dj(\tauijt)\zi(\tauijt)\|^2]$.
Finally, the proof of the lemma follows by using this inequality in (\ref{eqn:proof_of_lemma:innerproduct_bound2}) and following the same argument as of the proof of Lemma~\ref{lemma:innerproduct_bound1}.
\end{proof}

Now we can introduce our last lemma.

\begin{lemma}
\label{lemma:innerproduct_bound3}
    For any $t\geq B$, we have
    \begin{align*}
        &\E[-\frac{1}{n}\sum_{i,j}{\langle \nbli f(\x(\tauijt)),\Dj(\tauijt)\zi(\tauijt)\rangle}]\\
        &\qquad\qquad\leq-\frac{1}{3}\E[\|\nabla f(\x(t))\|^2]+4\sqrt{6}G^2Ld\sqrt{nB}p(t)\\
        &\qquad\qquad\qquad+\sqrt{6}nG^2L^2Bdp^2(t)+\frac{3}{2}u^2(t)L^2d.
    \end{align*}
\end{lemma}

\begin{proof}
Let $\Ft=\sigma(\x(\tauijt))$ be the $\sigma$-algebra generated by $\x(\tauijt)$.
Then, {using the law of total expectation (Theorem 4.1.13 in \cite{durrett2019probability}),} we have
\begin{align}
\label{eqn:proof_of_lemma:innerproduct_bound3}
    &\E[-\frac{1}{n}\sum_{i,j}{\langle \nbli f(\x(\tauijt)),\Dj(\tauijt)\zi(\tauijt)\rangle}]\cr
    =&\E[\E[-\frac{1}{n}\sum_{i,j}{\langle \nbli f(\x(\tauijt)),\Dj(\tauijt)\zi(\tauijt)\rangle} | \Ft]]\cr
    =&\E[-\frac{1}{n}\sum_{i,j}{\langle \nbli f(\x(\tauijt)),\E[\Dj(\tauijt)\zi(\tauijt) | \Ft]\rangle}]\cr
    =&\E[-\frac{1}{n}\sum_{i,j}{\langle \nbli f(\x(\tauijt)),\nbli\fuj(x(\tauijt))\rangle}]\cr
    =&\E[-\frac{1}{n}\sum_{i,j}{\langle \nbli f(\x(\tauijt))-\nbli f(\x(t)),\nbli\fuj(x(\tauijt))\rangle}]\cr
    &+\E[-\frac{1}{n}\sum_{i,j}{\langle \nbli f(\x(t)),\nbli\fuj(x(\tauijt))-\nbli\fuj(x(t))\rangle}]\cr
    &+\E[-\langle \nabla f(\x(t)), \nabla f^u(\x(t))-\nabla f(\x(t)) \rangle]\cr
    &-\E[\|\nabla f(\x(t))\|^2].
\end{align}

To bound the first term in (\ref{eqn:proof_of_lemma:innerproduct_bound3}), we use Lemma \ref{lemma:innerproduct_bound1} and Lemma \ref{lemma:innerproduct_bound2} and get
\begin{align}
\label{eqn:term1_in_lemma:innerproduct_bound3}
    &\E[-\frac{1}{n}\sum_{i,j}{\langle \nbli f(\x(\tauijt))-\nbli f(\x(t)),\nbli\fuj(x(\tauijt))\rangle}]\cr
    =&\E[-\frac{1}{n}\!\sum_{i,j}{\langle \nbli f(\x(\tauijt))\!-\nbli f(\x(t)),\!\!\Dj(\tauijt)\zi(\tauijt)\rangle}]\cr
    &\quad +\E[-\frac{1}{n}\sum_{i,j}\langle \nbli f(\x(\tauijt))-\nbli f(\x(t)),\cr
    &\qquad \nbli\fuj(\x(\tauijt))-\Dj(\tauijt)\zi(\tauijt)\rangle]\nonumber\\
    &\qquad\qquad\leq4\sqrt{6}G^2Ld\sqrt{nB}p(t).
\end{align}

To bound the second term in (\ref{eqn:proof_of_lemma:innerproduct_bound3}), we use $ab\leq \frac{1}{2}(a^2+b^2)$ and Lemma \ref{lemma:effect_of_delay} and get
\begin{align}
\label{eqn:term2_in_lemma:innerproduct_bound3}
    &\E[-\frac{1}{n}\sum_{i,j}{\langle \nbli f(\x(t)),\nbli\fuj(x(\tauijt))-\nbli\fuj(x(t))\rangle}]\cr
    \leq& \frac{1}{2n}\E\left[\sum_{i,j}\|\langle \nbli f(\x(t))\|^2\right]\cr
    &+\frac{1}{2n}\E\left[\sum_{i,j}\|\nbli\fuj(x(\tauijt))-\nbli\fuj(x(t))\|^2\right]\nonumber\\
    \leq&\frac{1}{2}\E[\|\nabla f(\x(t))\|^2]+\sqrt{6}nG^2L^2Bdp^2(t).
\end{align}

To bound the third term in (\ref{eqn:proof_of_lemma:innerproduct_bound3}), we use $ab\leq \frac{1}{2}(a^2+b^2)$ Lemma \ref{lemma:gradient_estimator} and get
\begin{align}
\label{eqn:term3_in_lemma:innerproduct_bound3}
    &\E[-\langle \nabla f(\x(t)), \nabla f^u(\x(t))-\nabla f(\x(t)) \rangle]\cr
    \leq& \frac{1}{2}\E\left[\frac{1}{3}\E[\|\nabla f(\x(t))\|^2]+3\|\nabla f^u(\x(t))-\nabla f(\x(t))\|^2\right]\cr
    \leq& \frac{1}{6}\E[\|\nabla f(\x(t))\|^2]+\frac{3}{2}u^2(t)L^2d.
\end{align}

Plugging (\ref{eqn:term1_in_lemma:innerproduct_bound3}), (\ref{eqn:term2_in_lemma:innerproduct_bound3}), and (\ref{eqn:term3_in_lemma:innerproduct_bound3}) into (\ref{eqn:proof_of_lemma:innerproduct_bound3}), we get
\begin{align*}
    &\E[-\frac{1}{n}\sum_{i,j}{\langle \nbli f(\x(\tauijt)),\Dj(\tauijt)\zi(\tauijt)\rangle}]\\
    \leq& -\frac{1}{3}\E[\|\nabla f(\x(t))\|^2]+4\sqrt{6}G^2Ld\sqrt{nB}p(t)\\
    &+\sqrt{6}nG^2L^2Bdp^2(t)+\frac{3}{2}u^2(t)L^2d.
\end{align*}

\end{proof}

{Now we are ready to prove the main result. 

\begin{proof}[Proof of Theorem~\ref{Theorem1}]
Our algorithm~\eqref{eqn:maindyn} can be compactly written as $\x(t+1)=\x(t)-\step(t)\grad(t)$, 
where $\grad(t)$ is the $d$-dimensional vector that is obtained by concatenation of $\grad^1(t),\ldots,\grad^n(t)$ vectors.
Using $L$-smooth condition (Lemma 5.7 in \cite{beck2017first}) of $f$, we have
\begin{align*}
    &f(\x(t+1)) \leq f(\x(t))-\! \langle \nabla f(\x(t)), \step(t)\grad(t) \rangle\! +\frac{L}{2}\| \step(t)\grad(t)\|^2 \\
    &= f(\x(t))-\step(t)\sum_{i=1}^n{\langle \nbli f(\x(t)), \grad^i(t) \rangle}+\frac{L}{2}\step^2(t)\|\grad(t)\|^2.
\end{align*}
Taking the expected value of both sides, we have
\begin{align}
\label{eqn:expected_iter}
    &\E[f(\x(t+1))] \leq \E[f(\x(t))]\\\nonumber &-\step(t)\E [\sum_{i=1}^n\langle \nbli f(\x(t)), \grad^i(t) \rangle] +\frac{L}{2}\step^2(t)\E[\|\grad(t)\|^2].
\end{align}


From Lemma \ref{lemma:innerproduct_bound1} and Lemma \ref{lemma:innerproduct_bound3} we have
\begin{align*}
    &\E [-\sum_{i=1}^n{\langle \nbli f(\x(t)), \grad^i(t) \rangle}]\\
    =&\E[-\frac{1}{n}\sum_{i,j}{\langle \nbli f(\x(t))-\nbli f(\x(\tauijt), \Dj(\tauijt)\zi(\tauijt) \rangle}]\\
    &+\E[-\frac{1}{n}\sum_{i,j}{\langle \nbli f(\x(\tauijt), \Dj(\tauijt)\zi(\tauijt) \rangle}]\\
    \leq & -\frac{1}{3}\E[\|\nabla f(\x(t))\|^2]+6\sqrt{6}G^2Ld\sqrt{nB}p(t)\\
    &+\sqrt{6}nG^2L^2Bdp^2(t)+\frac{3}{2}u^2(t)L^2d.
\end{align*}
Using the above inequality and Lemma \ref{lemma:2norm_grad_bound}, (\ref{eqn:expected_iter}) becomes
\begin{align*}
    &\E[f(\x(t+1))] \leq \E[f(\x(t))]-\frac{1}{3}\step(t)\E[\|\nabla f(\x(t))\|^2]\\
    &+6\sqrt{6}G^2Ld\sqrt{nB}\step(t)p(t)+\sqrt{6}nG^2L^2Bd\step(t)p^2(t)\\
    &+\frac{3}{2}u^2(t)L^2d\step(t)+\sqrt{6}G^2Ld\step^2(t).
\end{align*}
By taking the telescoping sum, we get
\begin{align}
\label{eqn:telesum}
    &\sum_{t=B}^T \step(t)\E[\| \nabla f(\x(t))\|^2] \cr
    \leq& 3\E[f(\x(B))-f^*]+18\sqrt{6}G^2Ld\sqrt{nB}\sum_{t=B}^T\step(t)p(t) \cr
    &+3\sqrt{6}nG^2L^2Bd\sum_{t=B}^T\step(t)p^2(t) \cr
    &+\frac{9}{2}L^2d\sum_{t=B}^Tu^2(t)\step(t)+3\sqrt{6}G^2Ld\sum_{t=B}^T\step^2(t).
\end{align}
Since $f$ is $G$-Lipschitz, we have
\begin{align*}
    \E[f(\x(B))] &\leq f(\x_0)+G\E[\|\x(B)-\x(0)\|] \\
    &\leq f(\x_0)+G\sum_{t=0}^{B-1}{\step(t)\E[\|\grad(t)\|]}\\
    &\leq f(\x_0)+\sqrt{2\sqrt{6}}G^2\sqrt{d}\sum_{t=0}^{B-1}{\step(t)}:=C_0.
\end{align*}
Therefore, using \eqref{eqn:telesum} and letting $S(T):=\sum_{t=B}^T \step(t)\E[\| \nabla f(\x(t))\|^2]$, we have
\begin{align*}
    S(T) \leq& 18\sqrt{6}G^2Ld\sqrt{nB}\sum_{t=B}^T\step(t)p(t) \cr
    &+3\sqrt{6}nG^2L^2Bd\sum_{t=B}^T\step(t)p^2(t) \cr
    &+\frac{9}{2}L^2d\sum_{t=B}^Tu^2(t)\step(t)+3\sqrt{6}G^2Ld\sum_{t=B}^T\step^2(t)+C_1,
\end{align*}
where $C_1=3(C_0-f^*)$ is a constant.

Since $\step(t)=\frac{\step_0}{(t+1)^\alpha}$ and $u(t)=\frac{u_0}{(t+1)^\beta}$, we have ${p(t)=\Oc(\frac{1}{t^\alpha})}$. Therefore, using Riemann sum approximation~\cite{keisler2013elementary}, we have
\begin{align}\label{eqn:ST}
    S(T)=
    \begin{cases}
    \Oc(T^{1-2\alpha}+T^{1-\alpha-2\beta}) \\
    \qquad\qquad\qquad \text{if } 0 < \alpha < \frac{1}{2} \text{ and } \alpha+2\beta < 1,\\
    \Oc(T^{1-2\alpha}) 
    \qquad \text{ if } 0 < \alpha < \frac{1}{2} \text{ and } \alpha+2\beta \geq 1,\\
    \Oc(T^{1-\alpha-2\beta}) \quad \text{ if } \frac{1}{2} \leq \alpha < 1 \text{ and } \alpha+2\beta < 1,\\
    \Oc(\log T) \quad\qquad \text{if } \alpha=\frac{1}{2} \text{ and } \beta \geq \frac{1}{4}, \\
    \qquad\qquad\qquad\quad \text{ or } \frac{1}{2} < \alpha < 1 \text{ and } \alpha+2\beta = 1,\\
    \Oc(1) \qquad\qquad \text{ if } \frac{1}{2} < \alpha < 1 \text{ and } \alpha+2\beta > 1.
    \end{cases} 
    \end{align}

Next, we establish an upper bound for 
\[M(T):=\frac{1}{T-B+1}\sum_{t=B}^T \E[\| \nabla f(x(t))\|^2].\]

For any $\theta \in (0,1)$,  define
\begin{align*}
    M_\theta(T)=\left[\frac{1}{T-B+1}\sum_{t=B}^T \left(\E[\| \nabla f(x(t))\|^2]\right)^\theta\right]^\frac{1}{\theta}.
\end{align*}
Note that by H\"older's inequality (Theorem 6.2 in \cite{folland1999real}), for any $p,q>1$ with ${\frac{1}{p}+\frac{1}{q}=1}$, and non-negative sequences $\{a_t\}_{t=1}^T$ and $\{b_t\}_{t=1}^T$, we have
\begin{align}\label{eqn:holder}
    \left(\sum_{t=1}^Ta_tb_t\right)^q \leq \left(\sum_{t=1}^T{a_t}^p\right)^{\frac{q}{p}}\sum_{t=1}^T{b_t}^q.
\end{align}
Let 
\begin{align*}
    a_t&=\left( \frac{1}{\step(t)} \right)^\theta=\frac{1}{{\step_0}^\theta}(t+1)^{\alpha\theta},\cr
    b_t&=\left(\step(t)\E[\| \nabla f(x(t))\|^2]\right)^\theta,
\end{align*}
and $(p,q)=\left(\frac{1}{1-\theta},\frac{1}{\theta}\right)$. Noting  $\sum_{t=B}^Tb_t^{\frac{1}{\theta}}=S(T)$, and using~\eqref{eqn:holder}, we have
\begin{align*}
    &M_\theta(T)=\left(\frac{1}{T-B+1}\sum_{t=B}^T{a_tb_t} \right)^\frac{1}{\theta}\\
    &\leq (T-B+1)^{-\frac{1}{\theta}}\left(\sum_{t=B}^T a_t ^\frac{1}{1-\theta} \right)^\frac{1-\theta}{\theta} \sum_{t=B}^T b_t ^\frac{1}{\theta}\\
    &=(T-B+1)^{-\frac{1}{\theta}}\left(\sum_{t=B}^T {\step_0}^{-\frac{\theta}{1-\theta}}(t+1)^\frac{\alpha\theta}{1-\theta} \right)^\frac{1-\theta}{\theta}S(T)\\
    &=\Oc(T^{-\frac{1}{\theta}}\left(T^{\frac{\alpha\theta}{1-\theta}+1}\right)^\frac{1-\theta}{\theta}S(T))
    =\Oc(T^{\alpha-1}S(T)).
\end{align*}
Using a similar argument as in the proof of Proposition 1 in \cite{reisizadeh2022dimix}, it can be shown that for any $\theta \in(0,1)$
\begin{align*}
    \sum_{t=B}^T \E[\| \nabla f(x(t))\|^2] \leq 
    \left[\sum_{t=B}^T \left(\E[\| \nabla f(x(t))\|^2]\right)^\theta\right]^\frac{1}{\theta}.
\end{align*}
Now for a given $\eps>0$, if we let $\theta=\frac{1}{1+\eps/2}$, we have
\begin{align*}
    &M(T) \leq (T-B+1)^{\frac{1}{\theta}-1}M_\theta(T)=(T-B+1)^\frac{\eps}{2} M_\theta(T)\\
    &=\Oc(T^{\alpha-1+\frac{\eps}{2}}S(T)).
\end{align*}
Therefore, using~\eqref{eqn:ST}, if $0<\eps/2<\min\{\alpha,2\beta,1-\alpha\}$, we have $M(T)\rightarrow 0$ as $T\rightarrow\infty$ for $0<\alpha<1$ and $\beta>0$. 

Finally, since the sum of the first $B$ terms is finite, we have $\lim_{T\to\infty}\frac{1}{T+1}\sum_{t=0}^{T}{\E[\|\nabla f(\x(t))\|^2]}=0$
and for a special case of $\alpha=\frac{1}{2}$ and $\beta=\frac{1}{4}$, we have 
\begin{align*}
    \frac{1}{T+1}\sum_{t=0}^{T}{\E[\|\nabla f(\x(t))\|^2]}&=\Oc(\log(T)T^{-(\frac{1}{2}-\frac{\eps}{2})})\cr 
    &=\Oc(T^{-\frac{1}{2}+\eps}).
\end{align*}
\end{proof}}

\section{Conclusion}

We studied a zeroth order cooperative multi-agent optimization problem with communication delay. We showed that for a class of diminishing smoothing radius and step sizes if the time delay between any two agents is uniformly bounded, the objective functions are $L$-smooth, Lipschitz, then our algorithm will converge to a first-order stationary point of the underlying problem, in particular, the established convergence rates are robust to the objective functions' smoothness parameters.



\bibliographystyle{ieeetr}
\bibliography{bib}

\end{document}